\documentclass[12pt,a4paper]{article}
\usepackage{indentfirst}
\usepackage{amsmath,amssymb,amsfonts,amsthm,graphics}
\usepackage{makeidx}
\usepackage{ifpdf}
\newtheorem{thm}{Theorem}

\newtheorem{lem}{Lemma}

\theoremstyle{definition}

\def\-{\mbox{--}}

\newtheorem{obs}{Observation}

\newtheorem{definition}{Definition}

\begin{document}
\title{Note on the upper bound of the rainbow index of a graph\Large\bf \footnote{Supported by NSFC No.11371205,
``973" program No.2013CB834204, and PCSIRT.}}
\author{\small Qingqiong~Cai, Xueliang~Li, Yan~Zhao\\
\small Center for Combinatorics and LPMC-TJKLC\\
\small Nankai University, Tianjin 300071, China\\
\small cqqnjnu620@163.com; lxl@nankai.edu.cn; zhaoyan2010@mail.nankai.edu.cn}
\date{}
\maketitle
\begin{abstract}
A path in an edge-colored graph $G$, where adjacent edges may be
colored the same, is a rainbow path if every two edges of it receive
distinct colors. The rainbow connection number of a connected graph
$G$, denoted by $rc(G)$, is the minimum number of colors that are
needed to color the edges of $G$ such that there is a rainbow path
connecting every two vertices of $G$. Similarly, a tree in $G$ is a
rainbow~tree if no two edges of it receive the same color. The
minimum number of colors that are needed in an edge-coloring of $G$
such that there is a rainbow tree connecting $S$ for each $k$-subset
$S$ of $V(G)$ is called the $k$-rainbow index of $G$, denoted by
$rx_k(G)$, where $k$ is an integer such that $2\leq k\leq n$.
Chakraborty et al. got the following result: For every $\epsilon>
0$, a connected graph with minimum degree at least $\epsilon n$ has
bounded rainbow connection number, where the bound depends only on
$\epsilon$. Krivelevich and Yuster proved that if $G$ has $n$
vertices and the minimum degree $\delta(G)$ then
$rc(G)<20n/\delta(G)$. This bound was later improved to
$3n/(\delta(G)+1)+3$ by Chandran et al. Since $rc(G)=rx_2(G)$, a
natural problem arises: for a general $k$ determining the true
behavior of $rx_k(G)$ as a function of the minimum degree
$\delta(G)$. In this paper, we give upper bounds of $rx_k(G)$ in
terms of the minimum degree $\delta(G)$ in different ways, namely,
via Szemer\'{e}di's Regularity Lemma, connected $2$-step dominating
sets, connected $(k-1)$-dominating sets and $k$-dominating sets of
$G$.


\end{abstract}

\section{Introduction}
All graphs considered in this paper are simple, finite and
undirected. We follow the terminology and notation of Bondy and
Murty \cite{Bondy}. Let $G=(V,E)$ be a nontrivial connected graph
with an edge-coloring $c: E\rightarrow \{1,2,\ldots,\ell\},$ $\ell
\in \mathbb{N}$, where adjacent edges may be colored the same. A
path of $G$ is a \emph{rainbow path} if no two edges of the path are
colored the same. The graph $G$ is \emph{rainbow connected} if for
every two vertices of $G$, there is a rainbow path connecting them.
The minimum number of colors for which there is an edge-coloring of
$G$ such that $G$ is rainbow connected is called the \emph{rainbow
connection number}, denoted by $rc(G)$. These concepts were
introduced by Chartrand et al. in \cite{Chartrand1}. Since it is
almost impossible to give the precise value of the rainbow
connection number for an arbitrary graph, many bounds for the
rainbow connection number have been given in terms of other graph
parameters, such as minimum degree and connectivity, etc. The
interested readers can see \cite{Caro,Chartrand1,Krivelevich and
Yuster,Sun,LiSun2}.

In \cite{Zhang}, Chartrand et al. generalized the concept of rainbow
path to rainbow tree. A tree $T$ in $G$ is a $rainbow~tree$ if no
two edges of $T$ receive the same color. For $S\subseteq V$, a
$rainbow\ $S-$tree$ is a rainbow tree connecting the vertices of
$S$. Given a fixed integer $k$  with $2\leq k \leq n$, an
edge-coloring $c$ of $G$ is called a $k$-$rainbow~coloring$ if for
every set $S$ of $k$ vertices in $G$, there exists a rainbow
$S$-tree. In this case, we called $G$ $k$-$rainbow~connected$. The
minimum number of colors that are needed in a $k$-$rainbow~coloring$
of $G$ is called the $k$-$rainbow~index$, denoted by $rx_k(G)$.
Clearly, when $k=2$, $rx_2(G)$ is exactly the rainbow connection
number $rc(G)$. For every connected graph $G$ of order $n$, it is
easy to see that $rc(G)\leq rx_3(G)\leq \cdots \leq rx_n(G)$. We
refer to \cite{Cai,Cai1,Chen,Liu} for more details about the
$k$-rainbow index.

Not surprisingly, as the minimum degree increases, the graph would
become more dense and therefore the rainbow connection number and
rainbow index would decrease. In \cite{Chakraborty},
\cite{Krivelevich and Yuster} and \cite{Chandran}, the authors
studied the relationship between the minimum degree $\delta(G)$ and
the rainbow connection number $rc(G)$:

\begin{thm}[\cite{Chakraborty}]\label{Chakraborty}
For every $\epsilon> 0$, a connected graph with minimum degree at
least $\epsilon n$ has bounded rainbow connection number, where the
bound depends only on $\epsilon$.
\end{thm}

\begin{thm}[\cite{Krivelevich and Yuster}]\label{Yuster}
If $G$ has $n$ vertices and the minimum degree $\delta(G)$ then $rc(G)<20n/\delta(G)$.
\end{thm}

\begin{thm}[\cite{Chandran}]
For every connected graph $G$ of order $n$ and minimum degree $\delta$,
$rc(G)\leq3n/(\delta+1)+3$.
Moreover, for every $\delta\geq2$, there exist infinitely many graphs $G$ such that $rc(G)\geq
3(n-2) /(\delta+1)-1$.
\end{thm}

Since $rc(G)$ is the case of $rx_k(G)$ for $k=2$, a natural problem
arises: for a general $k$ determining the true behavior of $rx_k(G)$
as a function of the minimum degree $\delta(G)$. In this paper, we
focus on this problem and obtain some upper bounds for $rx_k(G)$ in
terms of $\delta(G)$ in different ways, namely, via Szemer\'{e}di's
Regularity Lemma, connected $2$-step dominating sets, connected
$(k-1)$-dominating sets and $k$-dominating sets of $G$. The main
idea is similar to those of \cite{Chakraborty, Krivelevich and
Yuster, Chandran}. However, the proofs have their technical details
and the results are meaningful.

\section{Preliminaries}

For a graph $G$, we use $V(G)$, $E(G)$ and $\delta(G)$ to denote its
vertex set, edge set and minimum degree, respectively. For
$D\subseteq V(G)$, let $\overline{D}=V(G)\setminus D$, $|D|$ be the
number of vertices in $D$, and $G[D]$ be the subgraph of $G$ induced
by $D$. For two nonempty disjoint vertex subsets $X$ and $Y$ of a
graph $G$, let $E(X, Y)$ denote the set of edges of $G$ between $X$
and $Y$.

\begin{definition}
The distance between two vertices $u$ and $v$ in $G$, denoted by
$d(u,v)$, is the length of a shortest path between them in $G$. The
diameter of $G$, denoted by $diam(G)$, is the maximum distance
between every pair of vertices in $G$. The distance between a vertex
$v$ and a set $D\subseteq V(G)$ is $d(v,D):=\min\{d(v,u): u\in D\}$.
For a positive integer $k$, the $k$-step neighborhood of a set
$D\subseteq V(G)$ is $N^k(D):=\{x\in V(G): d(x,D)=k\}$. The distance
between two sets $X,Y\subseteq V(G)$ is $d(X,Y):=\min\{d(x,y): x\in
X, y\in Y\}$.
\end{definition}

\begin{definition}
The $Steiner~distance$ $d(S)$ of a set $S$ of vertices in $G$ is the
minimum size of a tree in $G$ containing $S$. Such a tree is called
a $Steiner$ $S$-$tree$ or simply a $Steiner~tree$. The
$k$-$Steiner~diameter$ $sdiam_k(G)$ of $G$ is the maximum Steiner
distance of $S$ among all sets $S$ with $k$ vertices in $G$.
\end{definition}

It is easy to get a simple upper bound and lower bound for $rx_k(G)$.

\begin{obs}[\cite{Zhang}]\label{obs1}
For every connected graph $G$ of order $n\geq 3$ and each integer $k$ with $3\leq k\leq n$,
$k-1\leq sdiam_k(G)\leq rx_k(G)\leq n-1$.
\end{obs}

\begin{definition}
Given a graph $G$ and a positive integer $k$,
a set $D\subseteq V(G)$ is called a \emph{k-step dominating set} of $G$,
if every vertex in $G$ is at a distance at most $k$ from $D$, i.e. $V(G)=D\cup(\cup_{i=1}^{k}N^i(D))$.
Further, if $G[D]$ is connected,
we call $D$ a {\it connected $k$-step dominating set} of $G$.
The {\it connected $k$-step domination number} $\gamma_{c}^k(G)$ is the number of
vertices in a minimum connected $k$-step dominating set of $G$.
When $k=1$, we may omit the qualifier ``1-step" in the above
names and the superscript $1$ in the notation.
\end{definition}

\begin{definition}
Given a graph $G$ and a positive integer $k$,
a set $D\subseteq V(G)$ is called a {\it $k$-dominating set} of $G$, if every vertex in
$\overline{D}$ is adjacent to at least $k$ distinct vertices of $D$.
Furthermore, if $G[D]$ is connected, we call $D$ a {\it connected $k$-dominating set}.
The connected $k$-domination number $\gamma_{k}^c(G)$ is the minimum cardinality among all the
connected $k$-dominating sets of $G$.
\end{definition}

\begin{definition}
Given a graph $G$ and a positive integer $k$, a dominating set $D$ of $G$ is called
a {\it $k$-way dominating set} if $d(v)\geq k$ for every vertex
$v\in \overline{D}$. In addition, if $G[D]$ is connected, we call
$D$ a {\it connected $k$-way dominating set}.
\end{definition}

\begin{definition}
Let $G$ be a graph and $D\subseteq V(G)$. For $v\in N^1(D)$, its
neighbors in $D$ are called {\it foots} of $v$, and the
corresponding edges are called {\it legs} of $v$.
\end{definition}

\begin{definition}
 An edge-colored graph is {\it rainbow} if no two edges in the graph share the same color.
\end{definition}

\section{Our results}

In this section, we will deduce some upper bounds of $rx_k(G)$ in
terms of the minimum degree $\delta(G)$ in different ways, namely,
via Szemer\'{e}di's Regularity Lemma, connected $2$-step dominating
sets, connected $(k-1)$-dominating sets and $k$-dominating sets of
$G$.

\subsection{An upper bound via Szemer\'{e}di's Regularity Lemma}

In this subsection, we will investigate the $k$-rainbow index of a
graph with the aid of Szemer\'{e}di's Regularity Lemma, and obtain
our first upper bound.

\begin{thm}\label{C}
For every $\epsilon >0$ and every fixed positive integer $k$, there
is a constant $C = C(\epsilon,k)$ such that if $G$ is a connected
graph with $n$ vertices and minimum degree at least $\epsilon n$,
then $rx_k(G)\leq C$.
\end{thm}

The proof of Theorem \ref{C} is based on a modified degree-form
version of Szemer\'{e}di's Regularity Lemma in \cite{Chakraborty}.
First of all, we need some more terminology and notation for stating
Szemer\'{e}di's Regularity Lemma.

Let $G$ be a graph and $X,\ Y$  be two subsets of $V(G)$. The
\emph{edge density} of the pair $(X, Y)$ is defined as $d(X, Y) =
|E(X, Y)|/(|X||Y|)$. A pair $(X, Y)$ is called
\emph{$\epsilon$-regular} if for every $X'\subseteq X$ and
$Y'\subseteq Y$ satisfying $|X'|\geq \epsilon|X|$ and $|Y'|\geq
\epsilon|Y|$, we have $|d(X', Y')-d(X, Y)|\leq \epsilon$. A
partition $V_1, \cdots, V_k$ of the vertex set of a graph $G$ is
called an \emph{equipartition} if $|V_i|$ and $|V_j|$ differ by no
more than 1 for all $1\leq i < j\leq k$. In particular, every $V_i$
has one of two possible sizes. The order of an equipartition denotes
the number of partition classes ($k$ above). An equipartition $V_1,
\cdots, V_k$ of the vertex set of a graph is called
\emph{$\epsilon$-regular} if all but at most $\epsilon{k \choose 2}$
of the pairs $(V_i, V_j)$ are $\epsilon$-regular. Szemer\'{e}di's
Regularity Lemma can be formulated as follows. In the sequel,
without mentioning explicitly, we assume that $\epsilon$ is a small
enough constant.

\begin{lem}[\cite{Szemeredi}]\label{lem1}
$($Regularity Lemma$)$ For every $\epsilon >0$ and positive integer
$K$, there exists $N = N_1(\epsilon, K)$, such that any graph with
$n\geq N$ vertices has an $\epsilon$-regular equipartition of order
$\ell$, where $K\leq \ell\leq N$.
\end{lem}

The modified degree-form version of the Regularity Lemma in
\cite{Chakraborty} comes to use in our proof as follows. Here
$[\ell]=\{1,2,\cdots,\ell\}$.

\begin{lem}[\cite{Chakraborty}]\label{regular}
$($Regularity Lemma-new version$)$ For every $\epsilon >0$ and
positive integer $K$ there is $N =N_2(\epsilon, K)$ such that the
following holds: If $G = (V, E)$ is a graph with $n> N$ vertices and
minimum degree at least $\epsilon n$ then there is a subgraph
$G^{''}$ of $G$, and a partition of $V$ into $V^{''}_1, \cdots,
V^{''}_{\ell}$ with the following properties:

1. $K\leq \ell\leq N$,

2. for all $i\in [\ell]$, $(1-\epsilon)\frac{n}{\ell}\leq |V^{''}_i|\leq (1+\epsilon^3)\frac{n}{\ell}$,

3. for all $i\in [\ell]$, $V^{''}_i$ induces an independent set in $G^{''}$,

4. for all $i,j\in [\ell]$, $(V^{''}_i, V^{''}_j)$ is an $\epsilon^3$-regular pair in $G^{''}$, with density either 0 or at least $\frac{\epsilon}{16}$,

5. for all $i\in [\ell]$ and every $v\in V^{''}_i$ there is at least one other class $V^{''}_j$ so that the number of neighbors of $v$ in $G^{''}$ belonging to $V^{''}_j$ is at least $\frac{\epsilon}{3\ell}n$.
\end{lem}

Given a graph $G=(V,E)$ and an edge coloring $c$ : $E \rightarrow \mathcal{C}$, let $\pi_c$ denote the corresponding partition of $E$ into at most $|\mathcal{C}|$ components.  For two edge-colorings $c$ and $c'$, $c'$ is \emph{a refinement} of $c$ if $\pi_{c'}$ is a
refinement of $\pi_c$, namely, $\pi_{c'}(e)=\pi_{c'}(e')$ always implies $\pi_c(e)=\pi_c(e')$.

\begin{obs}[\cite{Chakraborty}]\label{refinement}
Let $c$ and $c'$ be two edge-colorings of a graph $G$ such that $c'$ is a refinement of $c$. For any path $P$ in $G$, if $P$ is a rainbow path under $c$, then $P$ is a rainbow path under $c'$. In particular, if $c$ makes G rainbow connected, then so does $c'$.
\end{obs}

The following lemma bounds the number of edge-disjoint paths of length at most four between every two vertices in the same partition of a graph with some given property.

\begin{lem}\label{3.8}
For every $\epsilon >0$ and every fixed integer $k$, there exists $N =N_3(\epsilon,k)$ such that any graph $G = (V, E)$ with $n > N$ vertices and minimum degree at least $\epsilon n$ satisfies the following: there is a partition $\Pi$ of $V$ into $V_1,V_2,\cdots,V_{\ell}$ such that for every $i\in [\ell]$ and every two vertices $u, v \in V_i$, the number of edge-disjoint paths of length at most four from $u$ to $v$ is larger than $(3k)^4 \log n$.
\end{lem}

\begin{proof}
Given $\epsilon >0$ and a fixed integer $k$, let $L = N_1(\epsilon,1)$ and set $N$ to be the smallest number such that $\epsilon^4\frac{N}{L}> (3k)^4 \log N$. Now, given any graph $G = (V, E)$ with $n>N$ vertices and minimum degree at least $\epsilon n$, we apply Lemma \ref{regular} with parameters $\epsilon$ and $1$. The following proof is similar to that in \cite{Chakraborty}, so we omit it here.
\end{proof}

For a fixed integer $k$, we define a set of $4k$ distinct colors $\mathcal{C}=\{c^1_1,c^2_1,c^3_1,c^4_1,c^1_2,c^2_2,c^3_2,c^4_2,\cdots,\\
c^1_k,c^2_k,c^3_k,c^4_k\}$.
Given a coloring $c$: $E \rightarrow \mathcal{C}$,
a pair of vertices is called $c_i$-rainbow connected for some $i\in[k]$,
if there is a rainbow path between the vertices only using the colors $c^1_i$, $c^2_i$, $c^3_i$, $c^4_i$.
The following lemma is the key in the proof of Theorem \ref{C}.

\begin{lem}\label{3.6}
For every $\epsilon >0$ and every fixed integer $k$, there is $N =N_4(\epsilon,k)$ such that any connected graph $G = (V, E)$ with $n > N$ vertices and minimum degree at least $\epsilon n$ satisfies the following: there is a partition $\Pi$ of $V$ into $V_1, \cdots, V_{\ell}$ ($\ell \leq N$), and a coloring $c$ : $E \rightarrow \mathcal{C}$ such that for every $i\in [\ell]$ and every two vertices $u, v \in V_i$, there exist $k$ edge-disjoint rainbow paths between $u$ and $v$ under $c$.
\end{lem}
\begin{proof}
First we apply Lemma \ref{3.8} to get the partition $\Pi$.
Then we color every edge $e\in E$ with the colors in $\mathcal{C}$ uniformly and independently at random.
Observe that a fixed path $P$ of length at most four is a $c_j$-rainbow path with probability at least $\frac{4!}{(4k)^{4}}$ for each $j$ with $1\leq j\leq k$.
Thus a pair of vertices in $V_i$ is not all $c_j$-rainbow connected with probability at most $k\large(1-\frac{4!}{(4k)^{4}}\large)^{(3k)^4\log n}= o(n^{-2})$. It follows from the union bound that with positive probability, all such pairs are all $c_j$-rainbow connected for each $j$ with $1\leq j\leq k$. Hence the desired coloring must exist.
\end{proof}

\noindent\textbf{Proof of Theorem \ref{C}.} For a given $\epsilon >0$ and a fixed integer $k$, set $N =N_4(\epsilon,k)$ and $C =\frac{3N}{\epsilon}+4k$. Since $rx_k(G)\leq n-1$ by Observation \ref{obs1}, it follows that any connected graph $G=(V,E)$ with $n\leq C$ vertices satisfies $rx_k(G)\leq C$. So we assume that $n > C\geq N$. Let $\Pi$ be the partition of $V$ into $V_1,V_2, \cdots V_{\ell}$ from Lemma \ref{3.6}, where $\ell\leq N$.

Since the diameter of $G$ is bounded by $\frac{3}{\epsilon}$ in \cite{Chakraborty}, there is a connected subtree $T =(V_T,E_T)$ of $G$ with at most $\ell\cdot diam(G)\leq \frac{3}{\epsilon}N$ vertices such that for every $i\in [\ell]$, $V_T \cap V_i\neq \varnothing$.
For each $i\in[\ell]$, we choose one vertex from $V_T \cap V_i$ respectively, and call these vertices tree-nodes.
Let $c$ : $E \rightarrow \mathcal{C}$ be the coloring from Lemma \ref{3.6}, and let $\mathcal{H}=\{h_1, h_2, \cdots, h_{|E_T|}\}$ be a set of $|E_T|\leq \frac{3}{\epsilon}N$ different fresh colors. We refine $c$ by recoloring each edge $e_i\in E_T$ with the color $h_i\in \mathcal{H}$.
Let $c'$ : $E \rightarrow (\mathcal{C}\cup\mathcal{H})$ be the resulting coloring of $G$.
Next we will show this coloring $c'$ makes $G$ $k$-rainbow connected.

Suppose $S=\{v_1,v_2,\cdots,v_k\}$ is a set of $k$ vertices in $G$.
Without loss of generality, assume that $\{v_1,v_2,\cdots,v_q\}\subseteq V_T$,
and $\{v_{q+1},v_{q+2},\cdots,v_{k}\}\subseteq V\setminus V_T$ for some $0\leq p\leq k$.
For each $v_i$ with $q+1\leq i\leq k$, let $t(v_i)$ be the corresponding tree-node in the same component of $v_i$.
From Lemma \ref{3.6}, $v_i$ and $t(v_i)$ are connected by a rainbow path $P_i$ of length at most four,
which receives the colors from $\{c_i^1,c_i^2,c_i^3,c_i^4\}$ under the original coloring $c$.
Since $c'$ is a refinement of $c$, the path $P_i$ is still rainbow under $c'$.
By the definition of $c'$, there is a rainbow tree $T^*(\subseteq T)$ connecting the vertices
$v_1,\cdots,v_q,t(v_{q+1}),\cdots,t(v_k)$ using the colors from $\mathcal{H}$.
The paths $P_{q+1}, P_{q+2}, \cdots, P_k$ together with the tree $T^*$ induce
a connected rainbow subgraph $G^{*}$ of $G$ connecting $S$.
Thus there is a rainbow $S$-tree by generating a spanning tree of $G^{*}$.
Consequently, $rx_k(G)\leq |E_T|+4k\leq C$.

\subsection{An upper bound via connected $2$-step dominating sets}

In this subsection, we will continue the research on the $k$-rainbow index of a graph with the aid of connected $2$-step dominating sets, and obtain our second bound. First of all, we state the following lemma.

\begin{lem}[\cite{Krivelevich and Yuster}]\label{two spanning subgraphs}
A graph with minimum degree $\delta$ has two edge-disjoint spanning subgraphs, each with minimum degree at least $\lfloor(\delta-1)/2\rfloor$.
\end{lem}

\begin{thm}\label{thm2}
Let $G$ be a connected graph with $n$ vertices and minimum degree $\delta$.
Then $rx_k(G)< 10nk2^t/(\delta-2^{t+1}+2)-k-2$,
where $t$ is the integer such that $2^t\leq k<2^{t+1}$.
\end{thm}
\begin{proof}

The proof can be divided into the following three steps:

\noindent\textbf{Step 1:} Decompose a graph into $k$ edge-disjoint
spanning subgraphs.

\textbf{Claim 1:}\label{2ell spanning subgraphs}
A graph with minimum degree $\delta$ has $2^{\ell}$ edge-disjoint spanning subgraphs, each with minimum degree at least $(\delta-2^{\ell+1}+2)/2^{\ell}$.

\emph{Proof of Claim 1.} Let $G$ be a graph with minimum degree $\delta$. We apply induction on $\ell$. For $\ell=1$, it follows from Lemma \ref{two spanning subgraphs} that $G$ has two edge-disjoint spanning subgraphs, each with minimum degree at least $\lfloor(\delta-1)/2\rfloor\geq (\delta-2)/2$. Hence the assertion is true for $\ell=1$. Suppose the assertion holds up till $\ell-1$. Now we will prove it for $\ell$. By induction hypothesis, $G$ has $2^{\ell-1}$ edge-disjoint spanning subgraphs $G_1,G_2,\cdots, G_{2^{\ell-1}}$, each with minimum degree at least $(\delta-2^{\ell}+2)/2^{\ell-1}$. By Lemma \ref{two spanning subgraphs} again, each $G_i$ ($1\leq i\leq 2^{\ell-1}$) has two edge-disjoint spanning subgraphs, each with minimum degree at least $(\delta(G_i)-2)/2\geq (\delta-2^{\ell+1}+2)/2^{\ell}$. These $2^{\ell}$ spanning subgraphs of $G$ are our desired edge-disjoint spanning subgraphs.

Let $k$ be a positive integer. Then there exists an integer $t$ such that $2^t\leq k< 2^{t+1}$. Set $s=k-2^t$. Then we have the following claim:

\textbf{Claim 2:}\label{k spanning subgraphs}
A graph with minimum degree $\delta$ has $k$ edge-disjoint spanning subgraphs, $2s$ of which have minimum degree at least $(\delta-2^{t+2}+2)/2^{t+1}$, others have minimum degree at least $(\delta-2^{t+1}+2)/2^t$.

\emph{Proof of Claim 2.}
Let $G$ be a graph with minimum degree $\delta$.
By Claim 1, $G$ has $2^{t}$ edge-disjoint spanning subgraphs, each with minimum degree at least $(\delta-2^{t+1}+2)/2^{t}$. We select $s$ spanning subgraphs arbitrarily and replace each of them with its two edge-disjoint spanning subgraphs by Lemma \ref{two spanning subgraphs}.
Each of these $2s$ spanning subgraphs has minimum degree at least $(\delta-2^{t+2}+2)/2^{t+1}$. Therefore, these $2s$ spanning subgraphs together with the $2^t-s=k-2s$ non-selected spanning subgraphs are the ones we want.

Set $\alpha=(\delta-2^{t+1}+2)/2^t$ and $\beta=(\delta-2^{t+2}+2)/2^{t+1}$.
Using Claim 2, we have $k$ edge-disjoint spanning subgraphs of $G$,
denoted by $G_1,G_2,\cdots,G_{k-2s},G_{k-2s+1},\cdots,G_k$,
where $\delta_i:=\delta(G_i)\geq \alpha$ for $1\leq i\leq k-2s$,
and $\delta_i:=\delta(G_i)\geq \beta$ for $k-2s+1\leq i\leq k$.

\noindent\textbf{Step 2:} Construct connected 2-step dominating
sets.

\textbf{Claim 3:}
For each $G_i$ $(1\leq i\leq k)$, there exists a 2-step dominating set $D_i$ whose size is at most $n/(\delta_i+1)$.

\emph{Proof of Claim 3.}
Note that $G_i$ may be disconnected.
Let $C_{i1},C_{i2},\cdots,C_{ip}$ be all the connected components of $G_i$.
We execute the following process:

~~~~~~~~~~$D_i=\{v_{i1},v_{i2},\cdots,v_{ip}\}$, where $v_{ij}\in C_{ij}$.

~~~~~~~~~~While $N^3(D_i)\neq \emptyset$,

~~~~~~~~~~~~~~~~~~~~pick any $v\in N^3(D_i)$ and $D_i=D_i\cup\{v\}$.

Since the process ends only when $N^3(D_i)=\emptyset$, the final $D_i$ is a $2$-step dominating set of $G_i$.
Let $q$ be the number of iterations. Consider the cardinality of $D_i\bigcup N^1(D_i)$.
Initially, $|D_i\bigcup N^1(D_i)|\geq p(1+\delta_i)$.
In every iteration, we add a new vertex to $D_i$ and $|D_i\bigcup N^1(D_i)|$ increases by at least $1+\delta_i$.
Therefore, when the process ends, $(p+q)(1+\delta_i)\leq |D_i\bigcup N^1(D_i)| \leq n$, thus $|D_i|=p+q\leq n/(1+\delta_i)$.

\textbf{Claim 4:}
For each $D_i$ $(1\leq i\leq k)$, there exists a connected 2-step dominating set $D_i'(\supseteq D_i)$ whose size is at most $5n/(\delta_i+1)-4$.

\emph{Proof of Claim 4.}
Note that $G_i[D_i]$ may be  disconnected. Suppose that $Q_{i1},Q_{i2},\cdots,Q_{ir}$ are the connected components of $G_i[D_i]$.
Since $G$ is connected, every two distinct components $Q_{ij}$, $Q_{ij'}$ must be connected by a path in $G$.
Let $P_{jj'}$ be the shortest path between $Q_{ij}$ and $ Q_{ij'}$ in $G$.
Without loss of generality, we assume that $Q_{i1}$, $Q_{i2}$ are the two components at the minimum distance among all the pairs of components, i.e. $d(Q_{i1},Q_{i2})=\min\{d(Q_{ij},Q_{ij'}):1\leq j\neq j'\leq r\}$.
Then we claim $d(Q_{i1},Q_{i2})\leq5$.
Otherwise, we can find a vertex $v$ on $P_{12}$ such that $d(v,Q_{i1})\geq3$ and $d(v,Q_{i2})\geq3$.
Since $D$ is a 2-step dominating set, $d(v,Q_{ij})\leq2$ for some $Q_{ij}$,
thus $d(Q_{i1},Q_{ij})<d(Q_{i1},Q_{i2})$, a contradiction.
So by adding at most four vertices to $D_i$, we can reduce the number of components at least by 1.
Therefore, we can find a connected 2-step dominating set $D_i'(\supseteq D_i)$ by adding at most $4(r-1)$ vertices.
We have $|D_i'|\leq |D_i|+4(r-1)\leq|D_i|+4(|D_i|-1)=5|D_i|-4\leq 5n/(1+\delta_i)-4$.

Now consider $D'=D'_1\cup \cdots \cup D'_k$. Note that $G[D']$ may be disconnected.
Using the fact that each $D_i'$ ($1\leq i\leq k$) is a connected 2-step dominating set of $G$,
we can add at most $k-1$ vertices to $D'$ to obtain a set $D(\supseteq D')$
such that $G[D]$ is connected.
Moreover, we know $|D|\leq \sum_{i=1}^{k}|D_i'|+k-1\leq (k-2s)(5n/(\alpha +1)-4)+2s(5n/(\beta +1)-4)+k-1< 10nk2^t/(\delta-2^{t+1}+2)-3k-1$.

\noindent\textbf{Step 3:} Give a $k$-rainbow coloring.

For each $D_i$, let $U_i=\overline{D}\cap N^1(D_i)$ and $W_i=\overline{D}\cap N^2(D_i)$.
Then $U_i\cap W_i=\emptyset$ and $U_i\cup W_i=\overline{D}$, since $D_i$ is a 2-step dominating set of $G$.
For $1\leq i\leq k$, all the edges between $D_i$ and $U_i$ belonging to $G_i$ receive the same color $i$,
and  all the edges between $U_i$ and $W_i$ belonging to $G_i$ receive the same color $k+i$.
Choose $T$ as any spanning tree of $G[D]$, and color each edge of $T$ with a distinct fresh color.
All the remaining edges of $G$ are colored with $1$.
Then the total number of colors used is $|D|-1+2k< 10nk2^t/(\delta-2^{t+1}+2)-k-2$.

Next we will show that this edge-coloring makes $G$ $k$-rainbow connected.
Note that every vertex $v\in \overline{D}$ plays $k$ different roles
(i.e. for each $1\leq i\leq k$, $v$ is a vertex in either $U_i$ or $W_i$).
Thus $v$ has $k$ edge-disjoint rainbow paths $P^v_1,P^v_2,\cdots,P^v_k$,
where $P^v_i$ is a $v-D$ path of length 1 (if $v_i\in U_i$) or 2 (if $v_i\in W_i$) in $G_i$.
Moreover, if $P^v_i$ is of length 1, then the color of the edge on $P^v_i$ is $i$;
if $P^v_i$ is of length 2, then the colors of the two edges on $P^v_i$ are $i$ and $k+i$.
Let $S=\{v_1,v_2,\cdots, v_k\}$ be any set of $k$ vertices in $G$.
Without loss of generality, we assume $\{v_1,\cdots,v_p\}\subseteq D$
and $\{v_{p+1},\cdots,v_k\}\subseteq \overline{D}$ for some $0\leq p\leq k$.
For $v_i\in \overline{D}$, we choose the path $P^{v_i}_i$.
Clearly, $\bigcup_{i=p+1}^{k}P^{v_i}_i$ is still rainbow.
Then the paths $\{P^{v_i}_i: p+1\leq i\leq k\}$ together with the tree $T$ in $G[D]$
induce a rainbow $S$-tree.
Thus $rx_k(G)< 10nk2^t/(\delta-2^{t+1}+2)-k-2$.
\end{proof}

\subsection{Upper bounds via connected $(k-1)$-dominating sets and $k$-dominating sets}
In this subsection, we will turn to study the $k$-rainbow index of a graph
by connected $k$-dominating sets and $(k-1)$-dominating sets of the graph.

In the search toward good upper bounds for $rc(G)$ and $rx_3(G)$, an idea
that turned out to be successful more than once is considering the
``strengthened" connected dominating sets. Here we list some known results:

\begin{thm}[\cite{Chandran}]\label{thm5}
$(1)$ If $D$ is a connected two-way dominating set of a connected graph $G$,
then $rc(G)\leq rc(G[D])+3$.

$(2)$ If $D$ is a connected two-way two-step dominating set in a graph $G$, then
$rc(G)\leq rc(G[D])+6$.
\end{thm}

\begin{thm}[\cite{Liu}]\label{thm6}
Let $G$ be a connected graph with minimal degree $\delta(G)\geq 3$. If $D$ is a connected 2-dominating set of $G$, then $rx_{3}(G)\leq rx_{3}(G[D])+4$ and the bound is tight.
\end{thm}

\begin{thm}[\cite{Cai1}]\label{thm7}
$(1)$ If $D$ is a connected three-way dominating set of a connected graph $G$,
 then $rx_{3}(G)\leq rx_{3}(G[D])+6$. Moreover, the bound is tight.

$(2)$ If $D$ is a connected 3-dominating set of a connected graph $G$ with $\delta(G)\geq3$,
 then $rx_{3}(G)\leq rx_{3}(G[D])+3$. Moreover, the bound is tight.
\end{thm}

Next we will generalize these results to the $k$-rainbow index.

\begin{thm}\label{thm8}
$(1)$ Let $D$ be a connected $k$-dominating set of a connected graph $G$.
Then $rx_k(G)\leq rx_k(G[D])+k$, and thus $rx_k(G)\leq \gamma_k^c(G)+k-1$.

$(2)$ Let $D$ be a connected $(k-1)$-dominating set of a connected graph $G$ with minimum degree at least $k$.
Then $rx_k(G)\leq rx_k(G[D])+k+1$, and thus $rx_k(G)\leq \gamma_{k-1}^c(G)+k$.
\end{thm}

\begin{proof}
(1) Since $D$ is a connected $k$-dominating set, every vertex $v$ in
$\overline{D}$ has at least $k$ legs, denoted by $e^v_1, e^v_2,\cdots, e^v_k$.
For each $i$ $(1\leq i\leq k)$, we color the edge $e^v_i$ with the color $i$.
Let $r=rx_k(G[D])$.
Then we can color the edges in $G[D]$ with $r$ different colors from
$\{k+1,k+2,\cdots, k+r\}$ such that for every $k$ vertices in $D$,
there exists a rainbow tree in $G[D]$ connecting them.
If there remain uncolored edges in $G$, we color them with the color $1$.

Next we will show that this edge-coloring $c_1$ is a $k$-rainbow coloring of $G$.
Let $S=\{v_1,v_2,\cdots,v_k\}$ be any set of $k$ vertices in $G$.
Without loss of generality, we assume that
$\{v_1,\cdots, v_p\}\subseteq D$ and $\{v_{p+1},\cdots, v_k\}\subseteq \overline{D}$ for some $p$ ($0\leq p\leq k$).
For each $v_i\in \overline{D}$ $(p+1\leq i\leq k)$, let $f_i=v_iu_i$ be the leg of $v_i$ such that $c_1(f_i)=i$.
Then the edges $\{f_{p+1},\cdots,f_{k}\}$ together with the rainbow tree connecting the vertices
$\{v_1,\cdots,v_p,u_{p+1},\cdots,u_k\}$ in $G[D]$ induces a rainbow $S$-tree.
Thus $rx_k(G)\leq rx_k(G[D])+k$.
If we take a minimum connected $k$-dominating set $D^*$ in $G$, then $rx_k(G)\leq rx_k(G[D^*])+k\leq (|D^*|-1)+k=\gamma_k^c(G)+k-1$.

(2) Suppose $H$ is the subgraph of $G$ induced by $\overline{D}$.
Let $Z$ be the set of isolated vertices in $H$.
In every non-singleton connected component of $H$, we choose a spanning tree.
This gives a spanning forest on $\overline{D}\setminus Z$.
Choose $X$ and $Y$ as one bipartition defined by this forest.
Since $G$ has minimum degree at least $k$,
every vertex $v$ in $Z$ has at least $k$ legs, denoted by $e^v_1,\cdots, e^v_k$.
For each $i$ ($1 \leq i\leq k$), color $e^v_i$ with the color $i$.
Since $D$ is a $(k-1)$-dominating set,
every vertex in $X$ has at least $k-1$ legs, denoted by $f^v_1,\cdots, f^v_{k-1}$.
For each $i$ ($1 \leq i\leq k-1$), color $f^v_i$ with the color $i$ .
Similarly, every vertex in $Y$ has at least $k-1$ legs, denoted by $g^v_1,\cdots, g^v_{k-1}$.
For each $i$ ($1 \leq i\leq k-1$), color $g^v_i$ with the color $i+1$.
Further, we give each edge between $X$ and $Y$ the color $k+1$.
Let $r=rx_k(G[D])$.
Then we can color the edges in $G[D]$ with $r$ different colors from
$\{k+2,k+3,\cdots, k+r+1\}$ such that for every $k$ vertices in $D$,
there exists a rainbow tree in $G[D]$ connecting them.
If there remain uncolored edges in $G$, we color them with 1.

Next we will show that this edge-coloring $c_2$ is a $k$-rainbow coloring of $G$.
Let $S=\{v_1,v_2,\cdots,v_k\}$ be any set of $k$ vertices in $G$. Without loss of generality, we assume that
$\{v_1,\cdots,v_d\} \subseteq D$, $\{v_{d+1},\cdots,v_{d+z}\}\subseteq Z$,
$\{v_{d+z+1},\cdots,v_{d+z+x}\}\subseteq X$ and \\$\{v_{d+z+x+1},\cdots,v_{d+z+x+y}\}\subseteq Y$, where $d,z,x,y$ are non-negative integers such that $d+z+x+y=k$.

Case 1: $d+z+x=k$ (i.e. $S\cap Y=\emptyset$). Then for each $d+1\leq i\leq k-1$,
let $f_i=v_iu_i$ be the leg of $v_i$ such that $c_2(f_i)=i$.
For $i=k$, if $v_k\in Z$, then let $f_k=v_ku_k$ be the leg of $v_k$ such that $c_2(f_k)=k$;
if $v_k\in X$, then let $P_k=v_kw_ku_k$ $(w_k\in Y, u_k\in D)$ be the path such that $c_2(w_ku_k)=k,c_2(v_kw_k)=k+1$.

Case 2: $d+z+x=0$ (i.e. $S\subseteq Y$). Then for each $1\leq i\leq k-1$,
let $f_i=v_iu_i$ be the leg of $v_i$ such that $c_2(f_i)=i+1$.
For $i=k$, let $P_k=v_kw_ku_k(w_k\in X, u_k\in D)$ be the path such that $c_2(w_ku_k)=1,c_2(v_kw_k)=k+1$.

Case 3: $1\leq d+z+x\leq k-1$ (i.e. $S\cap (Z\cup X\cup D)\neq\emptyset$ and $S\cap Y\neq\emptyset$).
Then for each $d+1\leq i\leq k$,
let $f_i=v_iu_i$ be the leg of $v_i$ such that $c_2(f_i)=i$.

Then the edges $\{f_{d+1},\cdots,f_{k} (or\ P_k)\}$ together with the rainbow tree connecting the vertices
$\{v_1,\cdots,v_d,u_{d+1},\cdots,u_k\}$ in $G[D]$ induce a rainbow $S$-tree.
Thus $rx_k(G)\leq rx_k(G[D])+k+1$.
If we take a minimum connected $(k-1)$-dominating set $D^*$ in $G$, then $rx_k(G)\leq rx_k(G[D^*])+k+1\leq (|D^*|-1)+k+1=\gamma_{k-1}^c(G)+k$.
\end{proof}

In \cite{Caro2}, Caro, West and Yuster presented the following result for the connected $k$-domination number:

\begin{lem}[\cite{Caro2}]
Let $k$ and $\delta$ be positive integers satisfying $k<\sqrt{ln \delta}$
and let $G$ be a graph on $n$ vertices with minimum degree at least $\delta$.
Then $\gamma_k^c(G)\leq n\frac{ln \delta}{\delta}(1+o_\delta(1))$.
\end{lem}

Combining this lemma with Theorem \ref{thm8}, we come to the following conclusion:

\begin{thm}
Let $k$ and $\delta$ be positive integers satisfying $k<\sqrt{ln \delta}$
and let $G$ be a graph on $n$ vertices with minimum degree at least $\delta$.
Then $rx_k(G)\leq n\frac{ln \delta}{\delta}(1+o_\delta(1))$.
\end{thm}


\begin{thebibliography}{11}

\bibitem{Bondy} J.A. Bondy, U.S.R. Murty, \emph{Graph Theory},
GTM 244, Springer, 2008.

\bibitem{Cai} Q. Cai, X. Li, J. Song, \emph{Solutions to conjectures on the $(k,\ell)$-rainbow index of complete graphs}, Networks 62(2013), 220-224.

\bibitem{Cai1} Q. Cai, X. Li, Y. Zhao, \emph{The $3$-rainbow index and connected dominating sets}, arXiv preprint arXiv:1404.2377, 2014.

\bibitem{Caro}  Y. Caro, A. Lev, Y. Roditty, Z. Tuza, R. Yuster, \emph{On rainbow connection}, Electron. J. Combin. 15(1)(2008), R57.

\bibitem{Caro2} Y. Caro, D. West, R. Yuster, \emph{Connected domination and spanning trees with many leaves}, SIAM J. Discrete Math. 13(2000), 202-211.

\bibitem{Chakraborty} S. Chakraborty, E. Fischer, A. Matsliah,
\emph{Hardness and algorithms for rainbow connection} J. Comb. Optim. 21(3)(2011), 330-347.

\bibitem{Chandran} L. Chandran, A. Das, D. Rajendraprasad, N. Varma,
 \emph{ Rainbow connection number and connected dominating sets},
  J. Graph Theory 71(2)(2012), 206-218.

\bibitem{Chartrand1}  G. Chartrand, G. Johns, K. McKeon, P. Zhang, \emph{Rainbow connection in graphs}, Math. Bohem.  133(2008), 85-98.

\bibitem{Zhang}  G. Chartrand, F. Okamoto, P. Zhang, \emph{Rainbow trees in graphs and generalized connectivity}, Networks 55(2010), 360-367.

\bibitem{Chen} L. Chen, X. Li, K. Yang, Y. Zhao,
\emph{The $3$-rainbow index of a graph,} accepted by Discuss. Math.
Graph Theory.

\bibitem{Krivelevich and Yuster} M. Krivelevich, R. Yuster, \emph{The rainbow connection of a graph is (at most) reciprocal to its minimum degree},
J. Graph Theory, 63(3)(2010), 185-191.

\bibitem{Sun}  X. Li, Y. Shi, Y. Sun, \emph{Rainbow connections of graphs}:  A survey, Graphs \& Combin. 29(2013), 1-38.

\bibitem{LiSun2} X. Li, Y. Sun, \emph{Rainbow Connections of Graphs},
 SpringerBriefs in Math. Springer, New York, 2012.

\bibitem{Liu} T. Liu, Y, Hu, \emph{Some upper bounds for
$3$-rainbow index of graphs}, arXiv preprint arXiv:1310.2355, 2013.

\bibitem{Szemeredi} E. Szemer\'{e}di, \emph{Regular partitions of graphs}, Proc. Colloque Inter. CNRS 260 (CNRS, Paris) (1978), 399-401.

\end{thebibliography}
\end{document}